\newif\ifdraft
\drafttrue

\documentclass[12pt]{amsart}

\usepackage[margin=1.3in,marginpar=1in]{geometry}
\usepackage{enumitem}
\usepackage{xcolor}
\usepackage{url}
\usepackage{bbm}
\usepackage{amsthm}
\usepackage{amsmath}
\usepackage{amssymb}
\usepackage{tikz}
\usepackage{comment}
\usepackage{subcaption}   
\usepackage{mathtools}
\usepackage{soul}

\ifdraft
\usepackage{fancyhdr,datetime2}
\fancypagestyle{firstpage}{%
	
	\fancyhf{}
	\fancyfoot[C]{\thepage}
}
\fi

\numberwithin{equation}{section}

\newcommand{\eps}{\varepsilon}

\theoremstyle{plain}

\newtheorem{theorem}{Theorem}[section]
\newtheorem{lemma}[theorem]{Lemma}
\newtheorem{proposition}[theorem]{Proposition}

\newtheorem{example}[theorem]{Example}

\newtheorem{question}[theorem]{Question}

\theoremstyle{remark}
\newtheorem{remark}{Remark}

\author{Gregory Hemenway and Jason Tu}
\title[Relative Variational Principle for IFS]{A Relative Variational Principle for Expanding Iterated Function Systems}
\date{\today}

\address{Dept.\ of Mathematics, The Ohio State University, Columbus, OH 43210}
\email{hemenway.math@gmail.com}
\email{tu.403@buckeyemail.osu.edu}

\thanks{
This material is based upon work supported by the National Science Foundation under Award No.\ DMS-2316687.
}

\DeclareMathOperator{\diam}{diam}

\newcommand{\allA}{\ref{ass:degree}--\ref{ass:exact}}
\newcommand{\NN}{\mathbbm{N}}
\newcommand{\RR}{\mathbbm{R}}
\newcommand{\sS}{\mathbbm{S}^1}
\newcommand{\ZZ}{\mathbbm{Z}}

\newcommand{\PP}{\mathcal{P}}
\newcommand{\LL}{\mathcal{L}}

\newcommand{\MM}{\mathcal{M}}

\newcommand{\one}{\mathbbm{1}}
\newcommand{\uomega}{{\underline{\omega}}}

\begin{document}

\begin{abstract}
The variational principle is a key tool in the study of invariant measures for chaotic dynamical systems. In recent times, dynamicists have developed techniques in random and nonstationary systems to better model real-world phenomena. Here, we prove a relative variational principle for a class of expanding iterated function systems. In particular, we use a nonstationary
Ruelle--Perron--Frobenius theorem to show that the marginal entropy given an ergodic invariant measure on the symbolic base equals the average topological entropy along fibers in the induced skew product.
\end{abstract}

\maketitle

\ifdraft
\thispagestyle{firstpage}
\pagestyle{fancy}
\fi

\section{Introduction}
\label{intro.sec}

Let $(X,d)$ be a compact metric space and $T\colon X\to X$ a continuous dynamical system. Denote by $\mathcal{M}(X,T)$ the space of $T$-invariant probability measures on $X$. For a continuous potential $\varphi\colon X\to\RR$, the classical variational principle relates the topological pressure of $T$\footnote{an exponential growth rate of the number of distinct orbits} to the free energy of invariant measures
\[
    P(\varphi)
    =
    \sup\Big\{h_\nu(T)+\int\varphi\, d\nu\colon \nu\in\mathcal{M}(X,T)\Big\}
\]
where $h_\nu(T)$ is the measure-theoretic entropy of $T$. An invariant probability measure which achieves this supremum is called an \emph{equilibrium state}. When $\varphi\equiv0$, the variational principle reduces to
\[
        h_{\text{top}}(T)=\sup\Big\{h_\nu(T) \colon\ \nu\in\mathcal{M}(X,T) \Big\}.
    \]
In this case, an equilibrium state is called a \emph{measure of maximal entropy}.

For stationary expanding systems, equilibrium states are often studied through transfer operators and Gibbs-type estimates. In particular, it was shown by Walters \cite{W78}, Bowen \cite{B74} that expanding maps equipped with H\"older potentials admit unique equilibrium states.
In \cite{gH23}, the first author studied equilibrium states for nonuniformly expanding skew products and used fiberwise transfer operators to construct conditional measures along fibers.

\begin{theorem}[\cite{gH23}, see \cite{DG99} for the uniformly expanding case]
\label{thm:gH23}
Let $X$ and $Y$ be two compact metric spaces. Consider a continuous, expanding skew product 
$$T\colon X\times Y\to X\times Y, \qquad T(x,y)=(fx, g_x y).$$
For a H\"older potential $\varphi\colon X\times Y\to\RR$
and any $x\in X$, the fiberwise transfer operator is defined
\begin{equation}\label{eqn:transOp}
    L_x\colon C(Y_x)\to C(Y_{fx}),\qquad
    (L_x\psi)(z)=\sum_{g_x(y)=z} e^{\varphi(x,y)}\,\psi(y),
    \qquad z\in Y_{fx}
\end{equation}
where $Y_x:=\{x\}\times Y$ denotes fibers of $X\times Y$.

Then the following hold.
\begin{enumerate}

    \item There exists a H\"older continuous potential $\Theta\colon X\to\RR$.

\item There exists a family of probability measures $\{\nu_x\}_{x\in X}$ such that
    \[
    \nu_x(Y_x)=1
    \qquad\text{and}\qquad
    L_x^*\nu_{fx}=\lambda_x\nu_x
    \]
    where $\lambda_x=e^{\Theta(x)}$.

    \item The MME on $(X \times Y,T)$ is the measure $\mu
    = \int_X \mu_x\, d\mathbb{P}(x)$
    where $\mu_x$ is a normalization of $\nu_x$ and $\mathbb{P}$ is an equilibrium state for
    $\Theta$ on $(X,f)$.
\end{enumerate}
\end{theorem}

\vskip10pt 
Ledrappier and Walters established a relative variational principle for factor maps. Theorem \ref{thm:LW77} compares the marginal free energy of an invariant measure to the marginal average of fiberwise pressure quantities. This point of view is particularly natural for skew products and random dynamical systems, where ergodic base dynamics control fiberwise motion.

\begin{theorem}[{\cite[Theorem 2.1]{LW77}}]
\label{thm:LW77}
    Let $X$ and $Y$ be compact metric spaces and $T\colon X\to X,\ S\colon Y\to Y,\ \pi\colon X\to Y$ be continuous maps such that $\pi$ is surjective and $\pi \circ T=S \circ \pi$.
    Let $\varphi\colon X\to\RR$ be continuous and $\nu\in\mathcal{M}(Y,S)$. Then
    \[\text{sup}\big\{h_\mu(T|S)+\mu(\varphi)\colon \  \mu\in \mathcal{M}(X,T),\ \mu\circ\pi^{-1}=\nu\big\}=\int_{Y} P(T,\varphi,\pi^{-1} y)\ d\nu(y)\] where $P(T,\varphi,\pi^{-1}y)$ denotes the topological pressure of $\varphi$ for $T$ on  $\pi^{-1}y$. 
\end{theorem}

This raises a natural question about the fiberwise pressure of these fiber measures.

\begin{question}
    Do the fiber measures for an expanding skew product satisfy a relative variational principle?
\end{question}

In this paper, we will consider iterated function systems which can be viewed as a factor of a skew product. We will also restrict our attention to the zero potential.
Thus our setting should be viewed as a model case in terms of entropy to gain insight on the pressure problem for skew products.

Let $X$ be a compact metric space. Consider a family of continuous surjective maps $\Phi=\{T_0,\ldots,T_{\ell-1}\}$ where $T_j\colon X\to X$ are all surjective. Note that $\Phi$ induces a full two-sided shift on $\ell$ symbols $(\Sigma,\sigma)$.
An \emph{iterated function system (IFS)} is the dynamics encoded on $X$ by an element $\uomega\in\Sigma$ via the sequence of maps
\[
    T_\uomega^n
    =
    T_{\omega_{n-1}}\circ\cdots\circ T_{\omega_0}.
\]

Fix an ergodic $\sigma$-invariant probability measure $\mathbbm{P}$ on $\Sigma$. By Rokhlin's disintegration \cite{Rokhlin},
a probability measure $m$ on $\Sigma\times X$ with base marginal $\mathbbm{P}$ admits a disintegration
\[
    m=\int_\Sigma  m_\uomega\,d\mathbbm{P}(\uomega).
\]
It is easy to show that the measure $m$ is $T$-invariant if and only if its conditionals $\{m_\uomega\}$ satisfy the equivariance relation
$(T_{\omega_0})_*m_\uomega
    =
    m_{\sigma\uomega}$
for $\mathbbm{P}$-almost every $\uomega$. Thus, in the skew-product setting, invariance is not expressed by requiring one single measure on $X$ to be invariant under every generator. Instead, the natural object is an equivariant family of conditional measures over the fixed base system $(\Sigma,\sigma,\mathbbm{P})$.\\

We will consider IFSs that are of expanding type. Our precise assumptions are stated in Section~\ref{sec:setting}. For now, this roughly means that the generators are uniformly expanding and satisfy a fiberwise topological exactness\footnote{A map $T\colon X\to X$ is topologically exact if for every $\eps>0$, there exists a $N\in\NN$ such that for all $x\in X$, $T^n(B(x,\eps))=X$.} condition.

\theoremstyle{plain}
\newtheorem{theo}{\textbf{Theorem}}[theorem]
\renewcommand*{\thetheo}{\textbf{\Alph{theo}}}

\begin{theo}
\label{thmA}
Let $(X,\Phi)$ be an IFS of expanding type, and let $\mathbbm{P}$ be an ergodic $\sigma$-invariant probability measure on $\Sigma$. Let
    $T(\uomega,x)=(\sigma\uomega,T_{\omega_0}x)$
be the associated to $(X,\Phi)$ skew product. Then
\[
    h_{\text{top}}(T)
    =
    \int_\Sigma \log\lambda_\uomega\,d\mathbbm{P}(\uomega)
    =
    \sup\Big\{
        h_m(T)
        \colon
        m\in\mathcal{M}(\Sigma\times X,T), \pi_*m=\mathbbm{P}
    \Big\}
\]
where $\lambda_\uomega$ is the eigenvalue for the nonstationary transfer operator in Theorem \ref{thm:gH23}.
\end{theo}

Similar to \cite{LW77}, Theorem~\ref{thmA} shows that topological entropy of the skew product $(\Sigma\times X,T)$ is the average fiberwise entropy of $\Phi$ with respect to $\mathbb P$.\\

\vskip10pt \subsection{Related work and restrictions.}

Various relative variational principles have been proved for random dynamical systems; see Kifer \cite{K92}, Simmons--Urbanski \cite{S14}, and Stadlbauer, Suzuki, and Varandas \cite{SSV}. These results motivated the work in \cite{gH23} and \cite{CH25} and suggest that the constructed fiber measures should satisfy a relative variational principle.

Here, we restrict our attention in two ways:
\begin{enumerate}
    \item IFSs are only a specific realization of expanding skew products;
    \item we focus on the measure of maximal entropy case when $\varphi\equiv0$.
\end{enumerate}

These restrictions simplify the problem in an essential way. For a general potential, one must control both entropy and the Birkhoff sums of the potential, and the relevant quantity is pressure rather than entropy. In the zero-potential case, the fiberwise transfer operators count preimages, and their eigenvalues directly record the exponential growth rate of the fiberwise dynamics. This makes the measure of maximal entropy case the natural first step toward a more general pressure statement.\\

\noindent\textbf{Comparison with common invariant measures.}

One approach to the thermodynamic formalism of an iterated function system
$\Phi=\{T_0,\dots,T_{\ell-1}\}$ on $X$ seeks to find a single probability measure $\rho$
on $X$ invariant under every generator,
\[
  (T_j)_*\rho=\rho\qquad\text{for every }0\le j\le \ell-1 .
\]
We will denote by $\MM(X,\Phi)$ the space of probability measures which are invariant under every map in $\Phi$. Such measures $\rho$ exists whenever the generators commute, and its appeal is that it
produces an invariant product measure for the skew product
$T(\uomega,x)=(\sigma\uomega,T_{\omega_0}x)$ over the Bernoulli base
$(\Sigma,\sigma,\mathbb{P})$. Indeed, the product $m=\mathbb{P}\otimes\rho$ is
$F$-invariant if and only if $\rho$ is commonly invariant. In this case, the Abramov--Rohlin formula
\[
  h_m(T)=h_{\mathbb{P}}(\sigma)+h^{\mathrm{\sigma}}_m(T)
\]
reduces the fiber entropy $h^{\mathrm{f}}_m(T)$ to that of $\rho$ under the
generators, so that the entropy of $T$ can be computed explicitly (see \cite{AR62}).
However, this common-invariance requirement is very restrictive because it ignores the nonstationary nature of the skew product.

Under the right assumptions, the space $\MM(X,\Phi)$ is nonempty. However, the problem is it is far too small.
For $X=S^1$ with expanding generators
$$T_j\colon x\mapsto d_j x \pmod 1,$$
Lebesgue
measure is invariant under each $T_j$, so $\mathrm{Leb}\in \MM(X,\Phi).$
Yet in commuting higher-rank settings this is
essentially all it contains: when the generators are multiplicatively
independent (the $\times p,\times q$ situation), the rigidity conjectured by
Furstenberg \cite{Furstenberg1967}  (established by Rudolph \cite{Rudolph1990} and Johnson \cite{Johnson1992}) forces every atomless ergodic
measure of positive entropy invariant under the joint action to be Lebesgue.
Hence, once one restricts to atomless measures of positive entropy, the class collapses to a
single point,
\[
  \MM(X,\Phi)=\{\mathrm{Leb}\},
\]
and maximizing the entropy of measures in this class conveys no information about the fiberwise dynamics. The degeneration only worsens with rank, so the requirement is most
damaging precisely in the higher-rank case of greatest interest.

These difficulties stem entirely from insisting on one measure invariant under
all generators at once. Instead, we follow ideas from \cite{BC92}. The relative variational principle developed here
replaces common invariance with covariance along the base: instead of a fixed
$\rho$, we use the fibered family of measures $\{\nu_\uomega\}_{\uomega\in\Sigma}$ produced by
the fiberwise transfer operators in Theorem \ref{thm:gH23}.
No individual $\nu_\uomega$ need be invariant under any generator; the family
merely transforms covariantly along the orbits by
\[(T_{w_0})_*\mu_\uomega=\mu_{\sigma\uomega}.\]
Such families exist
in general, independently of any commutation hypothesis, and are insensitive to
the rigidity that collapses $\MM(X,\Phi)$.\\

\noindent\textbf{Organization of the paper.}

In Section~\ref{sec:setting}, we introduce the skew-product model for expanding IFSs, state our standing assumptions, define the fiberwise topological and metric entropy, and state the main theorem. We also discuss examples and recall the nonstationary Ruelle--Perron--Frobenius theorem used later in the proof. In Section~\ref{sec:wGibbs}, we prove the weak-Gibbs estimates for the fiberwise eigenmeasures for $L_\uomega$. In Section~\ref{sec:proof}, we use the local entropy formula to prove Theorem \ref{thmA}.

\section{Setting and Main Results}
\label{sec:setting}

\vskip10pt \subsection{Shift spaces}

Consider an alphabet $\mathcal{A}=\{0,\ldots,\ell-1\}$. Let $\Sigma=\mathcal{A}^\ZZ$ and $\Sigma^+=\mathcal{A}^{\NN\cup\{0\}}$, respectively, denote the space of bi-infinite and positively infinite sequences of elements of $\mathcal{A}$.
For any $\underline{x}, \underline{y} \in \Sigma$, we define their distance in the shift space as
\[
d(\underline{x}, \underline{y}) = 2^{-\min\{|k|\colon\ x_k \neq y_k\}}.
\]
This turns $\Sigma$ into a compact metric space.
The left shift map $\sigma\colon \Sigma\to\Sigma$ is defined by $$\sigma (\underline{x_0 x_1 \dots x_{n-1}})_k=\underline{x_1 \dots x_{n}}.$$

Finite collections of elements in $\mathcal{A}$ can be concatenated into words 
$\omega = \omega_0 \omega_1 \dots \omega_{n-1}.$
We call the set of all words that appear in some $x\in\Sigma$ the \emph{language} of $\Sigma$ (or $\Sigma^+$ respectively); i.e. 
$$\mathcal{L}(\Sigma) :=\bigcup_{n\ge 1}\mathcal{L}_n(\Sigma) =\{\omega_0\cdots\omega_{n-1}\in \mathcal{A}^n\}.$$
Given a word \( \omega_0 \omega_1 \dots \omega_{n-1} \) in the language, its corresponding {cylinder set} at position \( i \) is defined as:
\[
[\omega_0 \omega_1 \dots \omega_{n-1}]_i = \{ \underline{x} \in \Sigma \mid x_i = \omega_0, x_{i+1} = \omega_1, \dots, x_{i+n-1} = \omega_{n-1} \}
\]
That is, a cylinder at the $i^{th}$ position is the collection of all sequences \( \underline{x} \in \Sigma \) for which $\omega_0 \omega_1 \dots \omega_{n-1}$ appears in the $i^{th}$ position.

We say the shift $(\Sigma, \sigma)$ has {specification} if there exists $M \in \mathbb{N}$ such that for any $v,w\in\LL$, there is $u\in \LL_M$ such that $vuw\in\LL$.
Bowen's argument \cite{B74} shows that if $(\Sigma,\sigma)$ has specification, then there is a unique MME (also see \cite{gH_Masters}). Note that the unique MME for the full-shift is Bernoulli; i.e. assigns weights to cylinder sets according to a probability vector \(\Vec{\mathbf{p}} = (p_0, p_1, \ldots, p_{\ell-1})\) which satisfies \(p_j \ge 0\) and \(\sum_{j=0}^{\ell-1} p_j = 1\).

\vskip10pt \subsection{Expanding Iterated Function Systems}
Let $(X, d)$ be a compact metric space. 
Consider a collection of maps ${\Phi}=\{T_0,\ldots,T_{\ell-1}\}$
for which each $T_j\colon X\to X$ is a continuous, surjective map.
An \emph{iterated function system (IFS)} is a realization of a fiber in the skew product 
\[
T\colon \Sigma\times X \to \Sigma\times X,
\quad
T(\uomega,x)=(\sigma(\uomega),T_{\omega_0}(x)).
\]  
For every $\uomega\in\Sigma$, we define an orbit segment of length $n$ starting in $X_\uomega$ by
        \[
        T_{\uomega}^n=T_{\omega_{n-1}}\circ\ldots\circ T_{\omega_0}\colon X_\uomega\to X_{\sigma^n\uomega}
        \]
where $(\Sigma,\sigma)$ is the shift of $\ell$ symbols. We equip $\Sigma\times X$ with the $L^1$ metric; i.e. $d\bigl((\uomega,x), (\uomega',x')\bigr)
  = d_\Sigma(\uomega,\uomega') + d_X(x,x')$. 
We note that $d\bigl((\uomega,x), (\uomega,x')\bigr)
  = d_X(x,x')$.

For each $\uomega\in\Sigma$, we write $X_\uomega:=\{\uomega\}\times X$ for the fiber over $\uomega$. For $\epsilon>0$, the ball of radius $\epsilon$ centered at $x\in X_\uomega$ is
\[
    B(x,\epsilon;\uomega)
    :=
    \{y\in X_\uomega:d_X(x,y)<\epsilon\}.
\]
Under the identification $X_\uomega\simeq X$, this is simply the ordinary ball $B(x,\epsilon)\subset X$.
The image of this fiber ball under the fiber composition $T_\uomega^n$ is regarded as a subset of the target fiber $X_{\sigma^n\uomega}$.

We will assume that our IFS $\Phi=\{T_0,\ldots,T_{\ell-1}\}$ satisfies the following
standing assumptions.

\begin{enumerate}[label=\textup{(A\arabic*)}]
    \item\label{ass:degree} \textbf{Uniformly bounded degree:}
    There exists $D\in\mathbb N$ such that for every $\uomega\in\Sigma$ and every
    $z\in X_{\sigma\uomega}$,
    \[ \#\,(T_{\omega_0})^{-1}(z)\le D. \]

    \item\label{ass:expanding} \textbf{Uniformly expanding and locally onto:}
    There exist $\delta>0$ and $\rho\in(0,1)$ such that for every $\uomega\in\Sigma$
    and all $x,x'\in X_{\uomega}$ with $d_X(x,x')\le \delta$,
    \[ d_X(x,x')\le \rho\, d_X(T_{\omega_0}x,\,T_{\omega_0}x'). \]
    Moreover, $T_{\omega_0}\big(B(x,\delta;\uomega)\big)\supset
    B(T_{\omega_0}x,\delta;\sigma\uomega)$.

    \item\label{ass:exact} \textbf{Uniform fiberwise topological exactness:}
    For every $\epsilon>0$, there exists $N=N(\epsilon)\in\mathbb N$ such that for
    every $\uomega\in\Sigma$, every $x\in X_{\uomega}$, and every $n\ge N$,
    \[ T_{\uomega}^n\big(B(x,\epsilon;\uomega)\big)=X_{\sigma^n\uomega}. \]
\end{enumerate}
We say an IFS $\Phi$ is of \emph{expanding type} if it satisfies assumptions \allA. Note that these assumptions all hold uniformly in $\uomega\in\Sigma$.

\begin{example}[An Expanding IFS on the Circle]
\label{ex:ifs}
Let $\sS=\mathbb R/\mathbb Z$ and consider the IFS $\Phi=\{E_2,E_3\}$ where
\[
E_m:\sS\to \sS,\qquad E_m(x)=mx \mod 1,
\]
for $m\in\{2,3\}$.

We check that this IFS satisfies assumptions \allA.

For (A1), first note that each map $E_m$ has exactly $m$ preimages at every point $x\in X$. Since
$m\in\{2,3\}$, we may take $D=3$.

For (A2), choose $0<\delta<1/6$. If $d_{\sS}(x,y)<\delta$, then for
$m\in\{2,3\}$, we have
\[
d_{\sS}(E_m x,E_m y)
=
m\cdot d_{\sS}(x,y).
\]
So the maps are uniformly expanding with $\rho=1/2$. Since $E_m$ is a homeomorphism of $B(x,\delta)$ onto $B(E_m x,m\delta)$
\[
E_m(B(x,\delta))=B(E_m x,m\delta)\supset B(E_m x,\delta).
\]
Thus, both maps are locally onto.

For (A3), fix $\epsilon>0$. Write
\[
E_\uomega^n(x) =D_n(\uomega)x \mod 1,
\qquad
D_n(\uomega):=\prod_{i=0}^{n-1}\omega_i.
\]
Since $D_n(\uomega)\ge 2^n$, choose
$N=N(\epsilon)$ such that
$2^N(2\epsilon)>1.$
Then for every $n\ge N$, every $\uomega\in\Sigma$, and  $x\in \sS$,
the image
\[
E_\uomega^n(B(x,\epsilon))
=
E_{D_n(\uomega)}(B(x,\epsilon))
\]
covers all of $\sS$. Therefore the IFS is uniformly fiberwise topologically exact.
Hence, $E$ is of expanding type.
\end{example}

\vskip10pt \subsection{Entropy of an IFS}
\phantom{.}

\subsubsection{Topological Entropy}
In this section, we describe the concepts of entropy for IFSs that we will study.

Fix \( \uomega \in \Sigma \). For \( {x},{y} \in X_{\uomega}\), we define the {$n^{th}$-Bowen distance} on $X_\uomega$ as 
\[
d_{\uomega}^n ({x},{y}) = \max_{0 \leq k < n} d(T_{\uomega}^k x, T_{\uomega}^k y).
\]
We will call $B_n(x,\epsilon,\uomega) := \{(\uomega, y)\in X_\uomega:d_\uomega^n(x,y)<\epsilon\}$ the $n^{th}$-Bowen ball of radius $\eps$ centered at $x\in X_{\uomega}$.

We say a subset \( E \subset X_{\uomega} \) is {$(\uomega,n, \epsilon)$-separated} if for distinct \( x, y \in E \), $d_{\uomega}^n(x,y) > \epsilon$. Let \( {s}(\uomega,n, \epsilon) \) be the largest cardinality of an  {$(\uomega,n, \epsilon)$-separated} set in \( X_{\uomega} \). The quantities $s(\uomega,n,\epsilon)$
measure the number of orbit segments in the fiber $X_\uomega$ that are
distinguishable at scale $\epsilon$. Note that for any $\underline{x},\underline{y}\in [\omega_0 \omega_1 \dots \omega_{n-1}]$, we have $ {s}(\underline{x},n,\epsilon)= {s}(\underline{y},n,\epsilon)$. Hence, we denote this common value by $ {s}(\omega, n,\epsilon)$ where $\omega=\omega_0 \omega_1 \dots \omega_{n-1}$.

The fiberwise topological entropy along the itinerary $\uomega$ is
\[
h_{\mathrm{top}}(T,\uomega)
=
\lim_{\epsilon\to0}
\limsup_{n\to\infty}
\frac1n\log s(\uomega,n,\epsilon).
\]
The relative topological entropy of the skew product over the base measure
$\mathbb P$ is
\[
    h_{\mathrm{top}}(T)
    =
    \lim_{\epsilon\to0}
    \limsup_{n\to\infty}
    \frac1n\int_\Sigma \log s(\uomega,n,\epsilon)\,d \mathbb{P}(\uomega).
\]

\vskip10pt
Similarly, we say a set $F\subset X_{\uomega}$ is \emph{$(\uomega,n,\eps)$-spanning} if for every $y\in X_{\uomega}$ there is $z\in F$ with $d_{\uomega}^n(y,z)\le\eps$. Let $r(\uomega,n,\eps)$ be the least
cardinality of such a set. A \emph{$(\uomega,n,\eps)$-cover} is a cover of $X$ by sets $C$ such that $\diam T_{\uomega}^k(C)\le\eps$ for all $0\le k<n$. We let $N(\uomega,n,\eps)$ be the least cardinality of a $(\uomega,n,\eps)$-cover.
As with $s(\uomega,n,\eps)$, the counts $r(\uomega,n,\eps)$ and
$N(\uomega,n,\eps)$ depend only on the word $\omega_0\cdots\omega_{n-1}$, so we
write $r(\omega,n,\eps)$ and $N(\omega,n,\eps)$.

The three counts compare in the standard way:
\begin{lemma}
    For every $\uomega$, $n$, and $\eps>0$,
    \begin{equation}\label{eqn:sandwich}
      N(\uomega,n,2\eps)\ \le\ r(\uomega,n,\eps)\ \le\ s(\uomega,n,\eps)\ \le\ N(\uomega,n,\eps).
    \end{equation}
\end{lemma}
\begin{proof}
    Let $F$ be a minimal $(\uomega,n,\eps)$-spanning set. Then $X_{\uomega}\subset \bigcup_{z\in F} B_n(z,\eps,\uomega)$. That is, $F$ forms a 
    $(\uomega,n,2\eps)$-cover of $X_\uomega$ which gives $N(\uomega,n,2\eps)\le r(\uomega,n,\eps)$.
    
    Now consider a maximal $(\uomega,n,\eps)$-separated set $E\subset X_\uomega$. If there are $y\in X_\uomega\setminus E$ such that $d_\uomega^n(x,y)>\eps$ for all $x\in E$, then $E$ would not be maximal. Thus, $E$ is $(\uomega,n,\eps)$-spanning, which gives the middle inequality in \eqref{eqn:sandwich}.
    
    Finally, let $C$ be a $(\uomega,n,\eps)$-cover of $X_\uomega$. Suppose $|E|>|C|$. By the pigeonhole principle, there must be a set in $C$ which contains multiple elements of $E$. However, this contradicts $(\uomega,n,\eps)$-separation, which gives $s(\uomega,n,\eps)\le N(\uomega,n,\eps)$.
\end{proof}

In particular all three define the same entropy as $\eps\to0$.\\

The following lemma and proposition justify the passage from finite-time relative topological entropy to the $\mathbb P$-average of the pointwise fiberwise entropy. We first record the required form of the subadditive
ergodic theorem and then apply it to the covering numbers $N(\uomega,n,\eps)$.

\begin{lemma}[Subadditive ergodic theorem; cf.\ \cite{VO16}]\label{lem:kingman}
Let $\mathbb P$ be $\sigma$-invariant and let $f_n\colon\Sigma\to[0,\infty)$,
$n\ge1$, be measurable with $\int_\Sigma f_1\,d\mathbb P<\infty$ and subadditive
along the cocycle:
\begin{equation}\label{eqn:subadd}
  f_{n+m}(\uomega)\ \le\ f_n(\uomega)+f_m(\sigma^n\uomega)
  \qquad\text{for $\mathbb P$-a.e.\ }\uomega,\ \ \forall\,n,m\ge1 .
\end{equation}
Then there is a $\sigma$-invariant $f_\ast\in L^1(\mathbb P)$ with
$\frac1n f_n\to f_\ast$ both $\mathbb P$-a.e.\ and in $L^1(\mathbb P)$, and
\begin{equation}\label{eqn:interchange}
  \int_\Sigma f_\ast\,d\mathbb P
  =\lim_{n\to\infty}\frac1n\int_\Sigma f_n\,d\mathbb P
  =\inf_{n\ge1}\frac1n\int_\Sigma f_n\,d\mathbb P .
\end{equation}
If $\mathbb P$ is ergodic, then $f_\ast\equiv\int_\Sigma f_\ast\,d\mathbb P$ is
constant $\mathbb P$-a.e.
\end{lemma}

We omit the standard proof, which packages Kingman's subadditive ergodic theorem; see \cite{VO16}.

\begin{proposition}\label{prop:fiber-entropy}
Fix $\eps>0$. Then the following limit exists for $\mathbb P$-a.e.\ $\uomega\in\Sigma$ and is $\sigma$-invariant
\[
  h_{\mathrm{top}}(T,\uomega,\eps):=\lim_{n\to\infty}\frac1n\log N(\uomega,n,\eps).
\]
Consequently,
\[
  h_{\mathrm{top}}(T)
  =\lim_{\eps\to0}\ \lim_{n\to\infty}\frac1n\int_\Sigma\log N(\uomega,n,\eps)\,d\mathbb P(\uomega)
  =\int_\Sigma h_{\mathrm{top}}(T,\uomega)\,d\mathbb P(\uomega).
\]
Moreover, if $\mathbb P$ is ergodic, then 
$h_{\mathrm{top}}(T)=h_{\mathrm{top}}(T,\uomega)$
for $\mathbb P$-a.e.\ $\uomega$.
\end{proposition}

\begin{proof}
Since $X$ is compact, it admits a finite cover by balls of radius $\eps/2$, and such balls have diameter at most $\eps$. Thus $N(\uomega,1,\eps)$ is bounded by a constant independent of $\uomega$, which implies $f_1\in L^1(\mathbb P)$.

Let $\mathcal{A}$ be a minimal $(\uomega,n,\eps)$-cover of $X_{\uomega}$ and $\mathcal{B}$ a minimal $(\sigma^n\uomega,m,\eps)$-cover of $X_{\sigma^n\uomega}$. Then $\mathcal{C}=\{A\cap(T_{\uomega}^n)^{-1}B\colon A\in\mathcal A,\ B\in\mathcal B\}$ covers $X_{\uomega}$ and each
$C=A\cap(T^n_\uomega)^{-1}B$ has $(\uomega,n+m,\eps)$-diameter $\le\eps$:
for $0\le k<n$ we have $T^k_\uomega C\subset T^k_\uomega A$, while for
$k=n+j$ with $0\le j<m$ we have $T^{n+j}_\uomega C\subset
T^j_{\sigma^n\uomega}B$.
Therefore, we have
\begin{align*}
    N(\uomega,n+m,\eps) \le\ |\mathcal{C}|
    \le\ |\mathcal{A}||\mathcal{B}|= N(\uomega,n,\eps)N(\sigma^n\uomega,m,\eps).
\end{align*}
Hence, the sequence $f_n(\uomega):=\log N(\uomega,n,\eps)$ satisfies \eqref{eqn:subadd}.

For fixed $\eps>0$, Lemma~\ref{lem:kingman} gives a $\mathbb{P}$-integrable, $\sigma$-invariant limiting function $h_{\mathrm{top}}(T,\uomega,\eps):=\lim_{n\to\infty}\frac1n\log N(\uomega,n,\eps)$ such that
\[
\int_\Sigma h_{\mathrm{top}}(T,\uomega,\eps)\,d\mathbb P(\uomega) =\lim_{n\to\infty}\frac1n\int_\Sigma\log N(\uomega,n,\eps)\,d\mathbb P(\uomega).
\]
As $\eps\downarrow0$, the counts $N(\uomega,n,\eps)$ increase, so $\eps\mapsto h_{\mathrm{top}}(T,\uomega,\eps)$ is monotone and converges almost everywhere by \eqref{eqn:sandwich}. By monotone convergence, the limit in $\eps$
passes through the integral giving
\begin{align*}
    \int_\Sigma h_{\mathrm{top}}(T,\uomega)\,d\mathbb P(\uomega)
    &=\lim_{\eps\to 0}\int_\Sigma h_{\mathrm{top}}(T,\uomega,\eps)\,d\mathbb P(\uomega)\\
    &=\lim_{\eps\to 0}\lim_{n\to\infty}\frac1n\int_\Sigma\log N(\uomega,n,\eps)\,d\mathbb P(\uomega)\\
    &=\lim_{\eps\to 0}\limsup_{n\to\infty} \frac1n\int_\Sigma\log s(\uomega,n,\eps)\,d\mathbb{P}(\uomega) =h_{\mathrm{top}}(T).
\end{align*}
Moreover, since $\mathbb{P}$ is ergodic, Lemma \ref{lem:kingman} implies that for almost every $\uomega$, we have
\begin{align*}
    h_{\mathrm{top}}(T,\uomega)=\lim_{\eps\to 0} h_{\mathrm{top}}(T,\uomega,\eps)
    &=\lim_{\eps\to 0}\limsup_{n\to\infty} \frac1n\int_\Sigma\log s(\uomega,n,\eps)\,d\mathbb{P} (\uomega)=h_{\mathrm{top}}(T).\qedhere
\end{align*}
\end{proof}

\begin{remark}
We apply
the lemma to the covering count $N(\uomega,n,\eps)$ rather than to $S(\uomega,n,\eps)$ because exact
subadditivity can fail for separated sets: an $(\uomega,n+m,\eps)$-separated
set need not be $(\uomega,n,\eps)$-separated, since two points may remain
$\eps$-close during the first $n$ iterates and separate only afterwards.
Repairing this by comparing maximal separated and spanning sets costs a factor
of $2$ in the scale $\eps$, which destroys the cocycle relation
\eqref{eqn:subadd} at fixed $\eps$. The covering count is exactly
submultiplicative, and \eqref{eqn:sandwich} transfers the conclusion to $s(\uomega,n,\eps)$
and $r(\uomega,n,\eps)$.
\end{remark}

\label{ex:topEntropy}
\begin{example}
Recall the IFS $E=\{E_2, E_3\}$ from Example \ref{ex:ifs}. We will compute the topological entropy of $E$.

As above, we write $D_n(\uomega)=\prod_{i=0}^{n-1}\omega_i$ for cardinality of the set $(T_\uomega^n)^{-1}(y)$.
For sufficiently small $\epsilon>0$, we claim that $s(\uomega,n,\epsilon)$ grows like $D_n(\uomega)$. More precisely, we will show that 
there exists a constant $C_\epsilon\ge 1$, independent of $n$ and $\uomega$, such
that
\begin{equation}
    \label{eqn:separated}
    D_n(\uomega) \le s(\uomega,n,\epsilon) \le C_\epsilon D_n(\uomega).
\end{equation}
The lower bound is immediate since for $\epsilon>0$ sufficiently
small, distinct inverse branches are $(\uomega,n,\epsilon)$-separated.

For the upper bound, first note that \eqref{eqn:sandwich} implies
$s(\uomega,n,\eps)\ \le\ r(\uomega,n,\eps/2)$. We proceed by estimating the cardinality of $(\uomega,n,\eps/2)$-spanning sets.
Let $\mathcal{C}$ be a $\eps/2$-cover of $X_{\sigma^n\uomega}$ of minimal cardinality $C_\eps$. Note that the compactness of $X$ implies $C_\eps$ is finite
and is independent of $n$ and $\uomega$. 
Since $\eps<\delta$, we know that $E_{\uomega}^n$ is locally injective on $X_\uomega$. Thus, each $B\in\mathcal{C}$ has $D_n(\uomega)$ distinct preimage balls which are $(\uomega,n,\eps/2)$-Bowen balls since each $E_j$ is expanding.
Choosing one point from each set $(E_\uomega^n)^{-1}B$ for $B\in\mathcal{C}$ therefore produces a
$(\uomega,n,\eps/2)$-spanning set of cardinality 
$D_n(\uomega)\cdot C_\eps$. Hence
\[
  s(\uomega,n,\eps) \le r(\uomega,n,\eps/2)\le C_\eps\,D_n(\uomega).
\]

By \eqref{eqn:separated}, for every sufficiently small $\epsilon>0$ and all $n\in\NN$,
\[
  \frac1n\log D_n(\uomega)
  \le \frac1n\log s(\uomega,n,\epsilon)
  \le \frac1n\log D_n(\uomega)+\frac1n{\log C_\epsilon}.
\]
Since $C_\epsilon$ is independent of $n$, letting $n\to\infty$ gives
\[
  \limsup_{n\to\infty} \frac1n\log s(\uomega,n,\epsilon)
  =\limsup_{n\to\infty} \frac1n\log D_n(\uomega).
\]
The right-hand side does not depend on $\epsilon$, so
\[
  h_{\mathrm{top}}(T,\uomega)
  =\limsup_{n\to\infty}\frac1n\log D_n(\uomega)
  =\limsup_{n\to\infty}\frac1n\sum_{i=0}^{n-1}\log\omega_i .
\]
If $\mathbb P$ is ergodic, Birkhoff's ergodic theorem gives, for
$\mathbb P$-a.e.\ $\uomega$,
\[
  h_{\mathrm{top}}(T,\uomega)
  =\lim_{n\to\infty}\frac1n\sum_{i=0}^{n-1}\log\omega_i
  =\int_\Sigma \log\omega_0\,d\mathbb P(\uomega),
\]
in agreement with Proposition~\ref{prop:fiber-entropy}. In particular, if
$\mathbb P$ is the uniform Bernoulli measure on $\{2,3\}^{\mathbb N}$, then
\[
h_{\mathrm{top}}(T)
=\int_\Sigma \log \omega_0\,d\mathbb P(\uomega)
=\frac12\log2+\frac12\log3 =\log\sqrt6.
\]
\end{example}

\subsubsection{Metric Entropy}
\label{sec:metricEntropy}
\phantom{.}

Fix a \(\sigma\)-invariant probability measure \(\mathbb P\) on \(\Sigma\). Let $\pi: \Sigma \times X \to \Sigma$ a projection map $\pi_\Sigma(\uomega,x)=\uomega$ to the base. We denote the set of \(T\)-invariant probability measures with marginal \(\mathbb P\) by
\[
\mathcal M_{\mathbb P}(\Sigma\times X,T)
=
\left\{
\mu\in\mathcal M(\Sigma\times X,T):
\pi_{*}\mu=\mathbb P
\right\}.
\]

By Rokhlin's disintegration, each \(\mu\in\mathcal M_{\mathbb P}(\Sigma\times X,T)\) has a unique   disintegration into conditional measures $\{\mu_\uomega\}$ over \(\mathbb P\)-a.e. fibers $X_\uomega$
\[
\mu=\int_\Sigma \mu_\uomega\,d\mathbb P(\uomega).
\]
From this, it is easy to see that
\(T\)-invariance of \(\mu\) is equivalent to the fiberwise pseudo-invariance relation
\[
(T_{\omega_0})_*\mu_\uomega=\mu_{\sigma\uomega}\quad \text{ for } \mathbb P\text{-a.e. } \uomega.
\]

Let \(\xi\) be a finite measurable partition of \(X\) and $\nu\in\MM(X)$. As in \cite{BC92}, we define the (Shannon) entropy of $\xi$ as
\[
  H_{\nu}(\xi):=-\sum_{A\in\xi}\nu(A)\log\nu(A).
\]
Fix \(\mu\in\mathcal M_{\mathbb P}(\Sigma\times X,T)\). For \(\uomega\in\Sigma\), define $\xi_\uomega^n = \bigvee_{k=0}^{n-1}(T_\uomega^k)^{-1}\xi.$ The relative entropy of \(T\) with respect to the partition \(\xi\) is
\[
h_\mu(T,\xi)
=
\lim_{n\to\infty}
\frac1n
\int_\Sigma
H_{\mu_\uomega}\left(\xi_\uomega^n\right)
\,d\mathbb P(\uomega)
\]
In particular, $H_{\mu_\uomega}(\xi_\uomega^n)$ is the Shannon entropy of $\xi^n_\uomega$ with respect to the conditional $\mu_\uomega$. Thus, $h_\mu(T,\xi)$ is the average entropy along fibers. The metric entropy with respect to $\PP$ is $h_\mu(T)=\sup\{h_\mu(T,\xi)\colon \xi \text{ is a finite partition of $X$}\}$.\\

We use the following random Brin--Katok type entropy formula from Zhu~\cite{Zhu09}.

\begin{proposition}[Brin-Katok entropy formula \cite{Zhu09}]
\label{prop:randBrinKatok}
    Let $m\in\mathcal{M}(\Sigma\times X,T)$ with $h_m(T)<\infty$. If $\{m_\uomega\}$ is a disintegration of $m$ along the fibers $X_\uomega$, then for $\mathbb{P} $-a.e. $\uomega \in \Sigma$ and $m_\uomega$ -a.e. $x\in X$,$$h_m(T)=\lim_{\eps\to 0}\limsup_{n\to\infty} -\frac1n\log m_\uomega\big(B_n(x,\eps,\uomega)\big).$$  
\end{proposition}

\vskip10pt \subsection{Nonstationary Transfer Operators}

In this section, we recall the thermodynamic formalism of nonstationary measures from \cite{CH25}. 

The trajectories in $X$ encoded by $\uomega\in\Sigma$ form a nonstationary dynamical system
\[T_{\omega_k}\colon X_{\sigma^k \uomega}\to X_{\sigma^{k+1} \uomega}.
\]
Similar to \eqref{eqn:transOp} for skew products, we can define for $\varphi \equiv 0$ a nonstationary transfer operator
${L}_k\colon C(X_{\omega_k})\to C(X_{\omega_{k+1}})$ by
\begin{equation}\label{eqn:NStransOp}
    ({L}_{\omega_k}\psi)(x)
= \sum_{T_{\omega_k}(y)=x}\psi_{\omega_k}(y).
\end{equation}
The operator ${L}_{\uomega_k}$ is an average over preimages associated to the one-step map $T_{\omega_0}$.
Note that we will iterate these operators by
\[
L_\uomega^n=L_{\omega_{n-1}}\circ\ldots\circ L_{\omega_0}.
\]
Each of these operators has a dual operator
${L}_{\omega_k}^*\colon \mathcal{M}(X_{\sigma^{k+1}\uomega})\to \mathcal{M}(X_{\sigma^k\uomega})$
is defined by
\[
\int_{X_\uomega} \psi\, d({L}_\uomega^*\eta)
=
\int_{X_{\sigma \uomega}} {L}_\uomega\psi\, d\eta
\]
for every $\psi\in C(X_\uomega)$ and every $\eta \in\mathcal{M}(X_{\sigma\uomega})$. \\

We will use a zero-potential version of the nonstationary Ruelle--Perron--Frobenius theorem in \cite[Theorems 1.1 and 1.2]{CH25}, rewritten in our notation.

\begin{theorem}[{Nonstationary RPF theorem}]
\label{thm:nonstationary-rpf}
Suppose that for every $\uomega\in\Sigma$, the IFS $T_\uomega$ satisfies Assumptions \allA. 

Then the following hold.

\begin{enumerate}[label={\rm (\arabic*)}]
    \item There exist uniquely determined families of positive real numbers
    $\{\lambda_{\sigma^n\uomega}\}_{n\in\mathbb{Z}}$ and  probability measures
    $\{\nu_{\sigma^n\uomega}\}_{n\in\mathbb{Z}}$ on $X_{\sigma^n\uomega}$ such that
    \[
    L_{\sigma^n\uomega}^*
    \nu_{\sigma^{n+1}\uomega}
    =
    \lambda_{\sigma^n\uomega}
    \nu_{\sigma^n\uomega}.
    \]

    \item There is a unique family of strictly positive functions $\{h_{\sigma^n\uomega}\}_{\uomega\in\Sigma}$,
    such that
    \[
    \int h_{\sigma^n\uomega}\,    d\nu_{\sigma^n\uomega}=1
\quad \text{ and } \quad
    L_{\sigma^n\uomega}
    h_{\sigma^n\uomega}
    =
    \lambda_{\sigma^n\uomega}
    h_{\sigma^{n+1}\uomega}
    \]
    for every $n\in\mathbb Z$. Morever, there exists a constant $C>1$ (independent of
    $\uomega$, $n$, and $x$) such that
    \begin{equation}
    \label{itm:eigenfunctionbound}
        C^{-1} \le h_{\sigma^n\uomega}(x) \le C
    \quad \text{ for all } x\in X_{\sigma^n\uomega}.
    \end{equation}

    \item The measure
    $\mu_{\sigma^n\uomega}
    = h_{\sigma^n\uomega}
    \nu_{\sigma^n\uomega}$
    is a probability on the fiber $X_{\sigma^n\uomega}$ and satisfies the pseudo-invariance
    relation
    $(T_{\uomega_n})_*
    \mu_{\sigma^n\uomega}
    =
    \mu_{\sigma^{n+1}\uomega}.$
\end{enumerate}
\end{theorem}

\vskip10pt \subsection{Main Results}

In this section, we state our main results. First, we state an important technical theorem about the scaling of $(\uomega,n)$-Bowen balls. Its proof is given in Section \ref{sec:wGibbs}.

\begin{theorem}[Weak-Gibbs property]
\label{thm:weak-gibbs}
Let $(X,\Phi)$ be an IFS of expanding type. Let
$\{\lambda_\uomega\}_{\uomega\in\Sigma}$ and
$\{\nu_\uomega\}_{\uomega\in\Sigma}$ be the eigenvalues and eigenmeasures
given by Theorem~\ref{thm:nonstationary-rpf} for the associated skew product $(\Sigma\times X,T)$. 

Then for every $\epsilon\in(0,\delta]$, where $\delta$ is the constant from
assumption \ref{ass:expanding}, there exists a constant $\ell_\epsilon>0$ such
that for every $n\ge1$, every $\uomega\in\Sigma$, and every $x\in X$,
\[
    \ell_\epsilon
    \leq
    \frac{\nu_\uomega(B_n(x,\epsilon,\uomega))}
    {\exp(-S_n\log\lambda(\uomega))}
    \leq
    1,
\]
where $S_n\log\lambda(\uomega)    :=    \sum_{k=0}^{n-1}\log\lambda_{\sigma^k\uomega}.$
\end{theorem}

\vskip10pt
Now we state the general form of Theorem \ref{thmA} which we prove in Section \ref{sec:proof}.

\begin{theorem}
    Let $(X,\Phi)$ be an IFS of expanding type, and $\mathbb{P}$ be an ergodic $\sigma$-invariant probability measure on $\Sigma$. Let $\{\lambda_\uomega\}_{\uomega\in\Sigma}$ be the eigenvalues and eigenmeasures given by Theorem~\ref{thm:nonstationary-rpf} for the associated skew product $(\Sigma\times X,T)$. 
    
    Then
        \[
        h_{\text{top}}(T) =\int\log\lambda_\uomega\ d\mathbb{P}(\uomega)= \sup\{h_m(T)\colon m\in\mathcal{M}_\mathbb{P}(\Sigma \times X, T)\}
        \]
    where $h_m(T)$ is the metric entropy of $m$ on the skew product associated to $\Phi$.
\end{theorem}

\section{Weak-Gibbs Property}
\label{sec:wGibbs}

In this section, we prove Theorem \ref{thm:weak-gibbs} by establishing a weak-Gibbs property for the nonstationary eigenmeasures given by Theorem \ref{thm:nonstationary-rpf}. For convenience, we will write the product of eigenvalues along an orbit segment as
\[
\lambda_\uomega^n
:=
\prod_{k=0}^{n-1}
\lambda_{\sigma^k\uomega}.
\]

\begin{lemma}\label{lem:mapping}
    Let $\uomega\in\Sigma$, $x\in X_\uomega$, and $n\in\NN$. Fix $0<\epsilon\le\delta$, where $\delta$ is the constant from
\ref{ass:expanding}. The maps $T_{\uomega}^n$ maps $B_n(x,\eps,\uomega)$ homeomorphically onto the ball $B(T_{\uomega}^nx,\eps)$.
    Moreover, for every $0\leq k\leq n$ and $x, y \in B_n(x,\eps,\uomega)$,
    \[
    d_X(T^k_\uomega y, T^k_\uomega x)\leq \rho^{n-k} d_X(T^n_\uomega y, T^n_\uomega x)\quad \text{for all }  0\leq k < n.
    \]
\end{lemma}

\begin{proof}
We will denote the iterates along the orbit of $x=x_0$ as $x_k:=T_{\omega_0}^k(x)$.
Recall $x_k=T_{\uomega}^k(x)$. Since $\uomega$ is fixed, the
one-step map for the $k$-th iterate is the generator
$T_{\omega_k}\colon X_{\sigma^k\uomega}\to X_{\sigma^{k+1}\uomega}$.
By the locally-onto property of \ref{ass:expanding},
$T_{\omega_k}$ restricts to a homeomorphism in a neighborhood of $x_k$ onto
$B(x_{k+1},\epsilon)$. This gives a unique local inverse branch
\[
  g_k\colon B(x_{k+1},\epsilon)\to B(x_k,\epsilon),\qquad g_k(x_{k+1})=x_k,
\]
which is uniformly contracting: $d(g_k(a),g_k(b))\le \rho\,d(a,b)$ for all
$a,b\in B(x_{k+1},\epsilon)$  ($\rho\in(0,1)$ from \ref{ass:expanding}).
We can apply the same process for each $0\le k< n$. Hence, inductively, we have
\[
  d_X(T_{\uomega}^{\,k}a,\,T_{\uomega}^{\,k}b)\le \rho^{\,n-k}\,
  d_X(T_{\uomega}^{\,n}a,\,T_{\uomega}^{\,n}b).
\]

Define $G_{\uomega,n}:=g_0\circ g_1\circ\cdots\circ g_{n-1}\colon 
B(x_n,\varepsilon)\to X_\uomega.$
That is, for $y\in B(x_n,\varepsilon)$, $G_{\uomega,n}(y)$ is the preimage of $y$ under $T_\uomega^n$. For $y_n:=y\in B(x_n,\varepsilon)$, we can track the orbit of $y$ by writing
$
y_k:=g_k\circ g_{k+1}\circ\cdots\circ g_{n-1}(y)$ for each $0\le k\le n-1$.
Then for every $0\le k< n$,
$T_\uomega^k(G_{\uomega,n}(y))=y_k$ and
\[
d(y_k,x_k)\le \rho^{n-k}d(y_n,x_n)<\varepsilon.
\]
Therefore,
$G_{\uomega,n}(B(x_n,\varepsilon))\subset B_n(x,\varepsilon,\uomega).$

Conversely, if $u\in B_n(x,\varepsilon,\uomega)$, then
$T_\uomega^n u\in B(x_n,\varepsilon).
$
Moreover, every intermediate iterate $T_\uomega^k u$ lies in
$B(x_k,\varepsilon)$. Thus, $u=G_{\uomega,n}(T_\uomega^n u)$.
Hence,
\[
B_n(x,\varepsilon,\uomega)= G_{\uomega,n}(B(x_n,\varepsilon)).
\]
Thus,
$T_\uomega^n:B_n(x,\varepsilon,\uomega)\to B(x_n,\varepsilon)$
is bijective.
\end{proof}

\begin{theorem}
    \label{thm:wGibbs}
    Let $(X,\Phi)$ be an IFS of expanding type. Let
$\{\lambda_\uomega\}_{\uomega\in\Sigma}$ and
$\{\nu_\uomega\}_{\uomega\in\Sigma}$ be the eigenvalues and eigenmeasures
given by Theorem~\ref{thm:nonstationary-rpf} for the associated skew product $(\Sigma\times X,T)$. 

Then for every $\epsilon\in(0,\delta]$, where $\delta$ is the constant from
assumption \ref{ass:expanding}, there exists a constant $\ell_\epsilon>0$ such
that for every $n\ge1$, every $\uomega\in\Sigma$, and every $x\in X$,
\[
    \ell_\epsilon e^{-S_n\log\lambda(\uomega)}
    \leq
    \nu_\uomega(B_n(x,\epsilon,\uomega))
    \leq
    e^{-S_n\log\lambda(\uomega)}.
\]
\end{theorem}

\begin{proof}
Fix $\varepsilon\in(0,\delta]$, where $\delta$ is the constant from
assumption~\ref{ass:expanding}; this is the range for which the local inverse
branches of Lemma~\ref{lem:mapping} exist.
By Theorem~\ref{thm:nonstationary-rpf}, we know
$(L_\uomega^n)^*\nu_{\sigma^n\uomega}
= \lambda_\uomega^n\nu_\uomega$.
Since $\phi=0$, the operator $L_\uomega^n$ only counts preimages. Hence,
\begin{equation*}
L_\uomega^n\one_{B_n(x,\varepsilon,\uomega)}(y)
=
\#\Bigl((T_\uomega^n)^{-1}(y)\cap B_n(x,\varepsilon,\uomega)\Bigr)
=
\one_{B(x_n,\varepsilon)}(y),
\end{equation*}
by the bijectivity established in Lemma \ref{lem:mapping}. Estimating the measure of a ball in $X_{\sigma^n\uomega}$, we get
\begin{align*}
\nu_{\sigma^n\uomega}(B(x_n,\varepsilon))
&=
\int L_\uomega^n\one_{B_n(x,\varepsilon,\uomega)}\, d\nu_{\sigma^n\uomega}
=
\lambda_\uomega^n \nu_\uomega(B_n(x,\varepsilon,\uomega)).
\end{align*}
Equivalently, since $\nu_{\sigma^n\uomega}$ is a probability measure,
\begin{equation}\tag{3}
\frac{\nu_\uomega(B_n(x,\varepsilon,\uomega))}{e^{-\log\lambda_\uomega^n}}
=
\nu_{\sigma^n\uomega}(B(T_\uomega^n x,\varepsilon))\leq 1.
\end{equation}

By the fiberwise exactness from assumption \ref{ass:exact}, for fixed $\varepsilon>0$, there exists
$N(\varepsilon)\in\mathbb{N}$ such that for every base point $\uomega\in\Sigma$
and every fiber point 
$z\in X_{\sigma^n\uomega}$, we have
\begin{equation*}
T_{\sigma^n\uomega}^{N(\varepsilon)}(B(z,\varepsilon))
=
X_{\sigma^{n+N(\varepsilon)}\uomega}.
\end{equation*}
This means every point of $X_{\sigma^{n+N(\varepsilon)}\uomega}$ has at least
one preimage in $B(z,\varepsilon)$, hence
\[
L_{\sigma^n\uomega}^{N(\varepsilon)}\one_{B(z,\varepsilon)}(u)\ge 1
\qquad
\text{for all }u\in X_{\sigma^{n+N(\varepsilon)}\uomega}.
\]
Integrating over $\nu_{\sigma^{n+N(\varepsilon)}\uomega}$ gives
\[
1
=
\nu_{\sigma^{n+N(\varepsilon)}\uomega}(X_{\sigma^{n+N(\varepsilon)}\uomega})
\le
\int L_{\sigma^n\uomega}^{N(\varepsilon)}\one_{B(z,\varepsilon)}
\,d\nu_{\sigma^{n+N(\varepsilon)}\uomega}.
\]
Using the eigenmeasure relation for the word starting at $\sigma^n\uomega$, we obtain
\begin{equation}\tag{5}
\nu_{\sigma^n\uomega}(B(z,\varepsilon))
\ge
\bigl(\lambda_{\sigma^n\uomega}^{N(\varepsilon)}\bigr)^{-1}.
\end{equation}

Now take $z=T_\uomega^n x$. Combining \((3)\) and \((5)\) , we get
\[
\bigl(\lambda_{\sigma^n\uomega}^{N(\varepsilon)}\bigr)^{-1} \leq \nu_{\sigma^n\uomega}(B(T_\uomega^n x,\varepsilon))
=\frac{\nu_\uomega(B_n(x,\varepsilon,\uomega))}{e^{-\log\lambda_\uomega^n}}
\le 1.
\]
Therefore, if we let $L(\sigma^n\uomega,\varepsilon)
:= \bigl(\lambda_{\sigma^n\uomega}^{N(\varepsilon)}\bigr)^{-1},$
then
\[
0<L(\sigma^n\uomega,\varepsilon)
\le
\frac{\nu_\uomega(B_n(x,\varepsilon,\uomega))}{e^{-\log\lambda_\uomega^n}}
\le 1.
\]

To complete the proof, we show that $L(\sigma^n\uomega,\eps)$ is uniformly bounded away from zero (depending only on $\eps$). 

Indeed, by \ref{ass:degree}, there
exists $D\in\mathbb N$ such that $L_\uomega \mathbf 1(y)=\#T_\uomega^{-1}(y)\le D$
for every $\uomega\in\Sigma$ and every $y\in X_{\sigma\uomega}$. 
Using
$L_\uomega^*\nu_{\sigma\uomega}=\lambda_\uomega\nu_\uomega$
and evaluating both sides on the constant function $\mathbf 1$, we get
\[
\lambda_\uomega
=
\int L_\uomega\mathbf 1\,d\nu_{\sigma\uomega}
\le D.
\]
Therefore, every factor
satisfies $\lambda_{\sigma^{n+j}\uomega}^{-1}\ge D^{-1}$. Hence,

\[
\left(\lambda_{\sigma^n\uomega}^{N(\varepsilon)}\right)^{-1}
=\prod_{j=0}^{N(\varepsilon)-1}
\lambda_{\sigma^{n+j}\uomega}^{-1}
\ge
D^{-N(\varepsilon)}.
\]

Then we have for every $n\in\mathbb N$ and every $\uomega\in\Sigma$,
\[
L(\sigma^n\uomega,\varepsilon)=\left(\lambda_{\sigma^n\uomega}^{N(\varepsilon)}\right)^{-1}\ge 
D^{-N(\varepsilon)}=:\ell_\varepsilon>0.\qedhere
\]
\end{proof}

\section{Proof of Theorem \ref{thmA}}
\label{sec:proof}

In this section, we prove our main result. This is done in the following two propositions.

\begin{proposition}
\label{thm: right side of theorem A}
    Let $(X,\Phi)$ be an IFS of expanding type. Let $\mathbb{P}$ be an ergodic invariant probability measure on $\Sigma$ and $(\Sigma\times X,T)$ the induced expanding skew product. Let $\{\lambda_{\uomega}\}_{{\uomega} \in \Sigma}$ be the eigenvalues for the fiberwise transfer operator from Theorem \ref{thm:nonstationary-rpf}. Then
    $$\sup\big\{h_m(T)\colon\ m\in\MM(\Sigma \times X, T) \text{ with } m\circ\pi^{-1}=\mathbbm{P}\big\} \geq\int \log\lambda_{\uomega}\,d\mathbbm{P}(\uomega).$$
\end{proposition}

\begin{proof}
    Let $\mu=\int_\Sigma  \mu_{\uomega}\,
        d\mathbbm{P}(\uomega)$ where $\mu_{\uomega}=h_{\uomega}\nu_{\uomega}$ is as in Theorem \ref{thm:nonstationary-rpf}. The pseudo-invariance of $\{\mu_\uomega\}$ given by Theorem \ref{thm:nonstationary-rpf} makes $\mu$ T-invariant. 
    We also know that there exist constants
    $0<c<C<\infty$ such that for every measurable set
    $A\subset X_{\uomega}$,
    \[
        c\nu_{\uomega}(A)
        \leq
        \mu_{\uomega}(A)
        =
        \int_A h_{\uomega}\,d\nu_{\uomega}
        \leq
        C\nu_{\uomega}(A).
    \]
    Thus, we will verify $m$ inherits a weak-Gibbs property from $\nu_\uomega$.

    We estimate $\mu_{\uomega}(B_n(x,\eps,\uomega))$. By the upper bounds in Theorem \ref{thm:nonstationary-rpf} and Theorem \ref{thm:wGibbs}, for $\eps>0$
    sufficiently small,
    \[
        \mu_{\uomega}(B_n(x,\eps,\uomega))
        \leq
        C\nu_{\uomega}(B_n(x,\eps,\uomega))
        \leq
        C e^{-S_n\log\lambda(\uomega)}.
    \]
    Therefore,
    \[
        -\frac1n\log \mu_{\uomega}(B_n(x,\eps,\uomega))
        \geq
        \frac1n S_n\log\lambda(\uomega)-\frac{\log C}{n}.
    \]
    Taking the $\limsup$ on both sides, we get for $\mathbbm{P}$-a.e. $\uomega$ and $\mu_{\uomega}$-a.e. $x$,
    \[
        \limsup_{n\to\infty}
        -\frac1n\log \mu_{\uomega}(B_n(x,\eps,\uomega))
        \geq
        \int_\Sigma \log\lambda_{\uomega}\,d\mathbbm{P}(\uomega).
    \]
    Taking $\eps\to0$ and applying Proposition \ref{prop:randBrinKatok},
    we obtain
    \[
        h_\mu(T)
        \geq
        \int_\Sigma \log\lambda_{\uomega}\,d\mathbbm{P}(\uomega).
    \]
    Hence,
    \[
        \sup\big\{h_m(T)\colon\ m\in\MM(\Sigma \times X, T) \text{ with } m\circ\pi^{-1}=\mathbbm{P}\big\}
        \geq
        \int_\Sigma \log\lambda_{\uomega}\,d\mathbbm{P}(\uomega).
    \]
\end{proof}

\begin{proposition}\label{prop:top-upper}
Let $(X,\Phi)$ be an expanding IFS with induced skew product
$(\Sigma\times X,T)$ and ergodic base measure $\mathbb P$. Then
\[
  h_{\mathrm{top}}(T)\le
  \int_\Sigma \log\lambda_{\uomega}\,d\mathbb P(\uomega).
\]
\end{proposition}

\begin{proof}
Fix a small $\eps>0$. For $\mathbb P$-a.e.\ $\uomega$, let
$E\subset X_{\uomega}$ be a maximal $(\uomega,n,\eps)$-separated set, so
$\#E=s(\uomega,n,\eps)$. Since $x\neq y\in E$ implies $d_n^{\uomega}(x,y)>\eps$,
the Bowen balls $\{B_n(x,\eps/2,\uomega)\}_{x\in E}$ are pairwise
disjoint. By the lower bound of Theorem~\ref{thm:wGibbs} (at radius $\eps/2$),
\[
  \nu_{\uomega}\bigl(B_n(x,{\eps}/{2},\uomega)\bigr)
  \ge \ell_{\eps/2}\,e^{-S_n\log\lambda(\uomega)} .
\]
Summing over these disjoint balls and using $\nu_{\uomega}(X_{\uomega})=1$, we see
\begin{align*}
    1\geq \nu_{\uomega}\bigl(\bigcup_{x\in E} B_n(x,{\eps}/{2},\uomega)\bigr)
    &= \sum_{x\in E} \nu_{\uomega}\bigl(B_n(x,{\eps}/{2},\uomega)\bigr)\\
    & \geq \sum_{x\in E} \ell_{\eps/2}\,e^{-S_n\log\lambda(\uomega)} =s(\uomega,n,\eps) \ell_{\eps/2}\,e^{-S_n\log\lambda(\uomega)}.
\end{align*}
Hence, we have $s(\uomega,n,\eps)\le \ell_{\eps/2}^{-1}\,e^{S_n\log\lambda(\uomega)}$ so
\begin{align}
    \label{eqn:inequality}
    \frac1n\int \log s(\uomega,n,\eps)d\mathbb P(\uomega)
  \le \frac1n \int S_n\log\lambda(\uomega)d\mathbb P(\uomega)-\frac1n\log\ell_{\eps/2}.
\end{align}
Since $\mathbb P$ is $\sigma$-invariant, we have
$\int_\Sigma \log\lambda_{\sigma^k\uomega}\,d\mathbb P(\uomega)
=\int_\Sigma \log\lambda_{\uomega}\,d\mathbb P(\uomega)$ for every $k\ge0$.
Averaging over $0\le k<n$ gives
\[
\frac1n\int_\Sigma S_n\log\lambda(\uomega)\,d\mathbb P(\uomega)
=\int_\Sigma \log\lambda_{\uomega}\,d\mathbb P(\uomega)
\qquad\text{for every }n\ge1,
\]
while $\frac1n\log\ell_{\eps/2}\to0$ as $n\to\infty$. Taking
$\limsup_{n\to\infty}$ in \eqref{eqn:inequality} therefore yields
\[
  \limsup_{n\to\infty}\frac1n\int\log s(\uomega,n,\eps)\, d\mathbb P(\uomega)
  \le \int_\Sigma \log\lambda_{\uomega}\,d\mathbb P(\uomega).
\]
Letting $\eps\to0$ gives
$h_{\mathrm{top}}(T)\le\int_\Sigma\log\lambda_{\uomega}\,d\mathbb P(\uomega)$.
\end{proof}

Combining Propositions
\ref{thm: right side of theorem A} and \ref{prop:top-upper} gives
\[
    \int_\Sigma \log\lambda_{\uomega}\,d\mathbbm{P}(\uomega)
    \;\le\;
    \sup\big\{h_m(T)\colon\ m\in\MM_\mathbb{P}(\Sigma \times X, T)\big\}
    \;\le\;
    h_{\mathrm{top}}(T)
    \;\le\;
    \int_\Sigma \log\lambda_{\uomega}\,d\mathbbm{P}(\uomega),
\]
where 
the second inequality holds by definition of topological pressure. Since the two ends coincide, all three
inequalities are equalities, which proves Theorem \ref{thmA}.

\bibliographystyle{amsalpha}
\bibliography{Nonstationary}

\end{document}